\documentclass{amsart}
\usepackage{amssymb}
\usepackage{amsfonts}
\usepackage{amsmath}
\usepackage{graphicx}
\usepackage{amscd}

\theoremstyle{plain}
\newtheorem{acknowledgement}{Acknowledgement}

\newtheorem{definition}{Definition}
\newtheorem{example}{Example}

\newtheorem{lemma}{Lemma}

\newtheorem{proposition}{Proposition}

\numberwithin{equation}{section}

\input{tcilatex}

\begin{document}
\title[\textbf{the best simultaneous approximation in linear 2-normed spaces}%
]{\textbf{The best simultaneous approximation in linear 2-normed spaces} }
\author[\textbf{M. Acikgoz}]{\textbf{Mehmet Acikgoz}}
\address{\textbf{University\ of\ Gaziantep Faculty\ of\ Science and\ Arts,\
Department of\ Mathematics 27310 Gaziantep,\ TURKEY }}
\email{\textbf{acikgoz@gantep.edu.tr}}
\date{\textbf{May 07, 2012}}

\begin{abstract}
In this paper, we shall investigate and analyse a new study on the best
simultaneous approximation in the context of linear 2-normed spaces inspired
by Elumalai and his coworkers in \cite{Elumalai1}. The basis of this
investigation is to extend and refinement the definition of the classical
aproximation, best approximation and some related concepts to linear
2-normed spaces.
\end{abstract}

\maketitle

\section{\textbf{Introduction}}

The problem of best and simultaneous best approximation has been studied by
several mathematicians (for more informations, see \cite{Cheney}, \cite%
{Cobzas}, \cite{Diaz1}, \cite{Diaz2}, \cite{Dunham}, \cite{Lorentz}, \cite%
{Goel1}, \cite{Goel2}, \cite{Singer}). Most of these works have dealt with
the existence, uniqueness and characterization of best approximations in
spaces of continuous functions with values in Banach spaces. Recently, many
works on approximation has been done on 2-structures such as 2-normed
spaces, generalized 2-normed spaces (for details, see \cite{Chen}, \cite%
{Lewandowska1}) and 2-Banach spaces (see \cite{Acikgoz2}, \cite{Acikgoz3}, 
\cite{Elumalai1}, \cite{Elumalai2}, \cite{Elumalai3}, \cite{Elumalai4}, \cite%
{Elumalai5}, \cite{Ehret}). Diaz and McLaughlin \cite{Diaz1} and Dunham \cite%
{Dunham} have considered simultaneously approximating of two real-valued
continuously functions $f_{1}$, $f_{2}$ defined on $[a,b]$, by elements of
set $C[a,b].$ Several results of best simultaneous approximation in the
context of linear space were obtained by Goel \textit{et al. }(for details,
see \cite{Goel1}, \cite{Goel2}). The subject of approximation theory has
attracted the attention of several mathematicians during the last 130 years
or so. This theory is now an extremely extensive branch of mathematical
analysis. It has many applications in many areas, especially in engineering.

The concept of linear 2-normed spaces has been investigated by Gahler in
1965 \cite{Gahler} and given many important properties and examples for
these spaces. After, these spaces have been developed extensively in
different subjects by other researchers from many points of view and then
the field has considerably grown. Z. Lewandowska published some of papers on
2-normed sets and generalized 2-normed spaces in 1999-2003 (see \cite%
{Lewandowska1}). In \cite{Elumalai1}, Elumalai and his coworkers published a
series of papers related this subject. They have developed best
approximation theory in the context of linear 2-normed spaces. These are
some works on characterization of 2-normed spaces, extension of
2-functionals and approximation in 2-normed spaces (see \cite{Acikgoz1}, 
\cite{Acikgoz2}). Also, the author has some works in $\varepsilon $%
-approximation theory \cite{Acikgoz3} and Rezapour has also such studies in 
\cite{Rezapour1}. The essential aim of this paper is to derive new different
definitions of approximation and obtain some results related to these
definitions. The essential results of the set of best simultaneous
approximation are given in the fourth section of this paper.

Throughout this paper, we first fix some notations. Let $X$ be a linear
space and $L\left\{ y\right\} $ be the subspace of $Y$ generated by $y$. It
is also let $\left( X,\left\Vert .\right\Vert \right) $ and $(X,\Vert
.,.\Vert )$ denote a normed space and $2$-normed space with the
corresponding norms, respectively. $%
\mathbb{R}
$ denotes the set of real numbers, $%
\mathbb{N}
$ denotes the set of natural numbers and $%
\mathbb{C}
$ denotes the set of complex numbers. Throughout this work, $K$ is variously
considered as an indeterminate, as a real number $K\in 
\mathbb{R}
$, or as a complex number $K\in 
\mathbb{C}
$.

We now summarize our work in four section as follows:

In first section, we gave history of normed and $2$-normed spaces and
motivation of our work. In section 2 and 3, we specify definitions and
properties of normed and $2$-normed spaces, respectively. In section 4, we
gave suitable a definition for studying in linear $2$-normed spaces and so
we derived three lemma and two proposition by using our definition.

Thus, we are now ready in order to begin with the second section as follows.

\section{\textbf{Some Definitions of Normed Spaces}}

\begin{definition}
Let $\left( X,\left\Vert .\right\Vert \right) $ be a normed space and $%
K\subset X$. For $u\in X$,

\begin{equation*}
\underset{v\in K}{\inf }\left\{ \left\Vert u-v\right\Vert \right\}
\end{equation*}%
is a general best approximation.
\end{definition}

Mohebi and Rubinov (\cite{Mohebi1}) and Rezapour in \cite{Rezapour1}\ gave
the main preliminaries on the approximation theory in the usual sense as
follows:

\begin{definition}
Let $\left( X,\left\Vert .\right\Vert \right) $ be a linear normed space.
For a nonempty subset $A$ of $X$ and $x\in X$,

\begin{equation*}
d\left( x,A\right) =\underset{a\in A}{\inf }\left\{ \left\Vert
u-a\right\Vert \right\}
\end{equation*}

denotes the distance from $x$ to the set $A$. If 
\begin{equation*}
\left\Vert x-a_{0}\right\Vert =d\left( x,A\right) \text{.}
\end{equation*}

Then, we say that a point $a_{0}\in A$ is called a best approximation for $%
x\in X$. If each $x\in X$ has at least one best approximation $a_{0}\in A,$
then $A$ is called a proximinal subset of $X$. If each $x\in X$ has a unique
best approximation $a_{0}\in A$, then $A$ is called a Chebyshev subset of $X$%
.
\end{definition}

\begin{definition}
Let $A\subset X$. For $x\in X$,%
\begin{equation*}
P_{A}\left( x\right) =\left\{ a\in A:\left\Vert x-a\right\Vert =d\left(
x,A\right) \right\}
\end{equation*}%
where $P_{A}\left( x\right) $, the set of all best approximations of $x$ in $%
A$. We know that $P_{A}\left( x\right) $ is a closed and bounded subset of $%
X $. For $x\notin A$, $P_{A}\left( x\right) $ is located in the boundary of $%
A$.
\end{definition}

\begin{definition}
Let $\left( X,\left\Vert .\right\Vert \right) $ be a linear normed space.
For a nonempty subset $A$ of $X$ and a nonempty set $W$ of $X$,%
\begin{equation*}
d\left( A,W\right) =\underset{w\in W}{\inf }\underset{a\in A}{\sup }\left\{
\left\Vert a-w\right\Vert \right\}
\end{equation*}%
denotes the distance from the set $A$ to the set $W$. If%
\begin{equation*}
\underset{w\in W}{\inf }\underset{a\in A}{\sup }\left\{ \left\Vert
a-w\right\Vert \right\} =\underset{a\in A}{\sup }\left\{ \left\Vert
a-w_{0}\right\Vert \right\} \text{.}
\end{equation*}%
Then, we say that a point $w_{0}\in W$ is called a best approximation from $%
A $ to $W$.
\end{definition}

\section{\textbf{Properties of }$2$\textbf{-Normed Spaces}}

In \cite{Freese}, Cho \textit{et al}. defined linear 2-normed spaces and
gave interesting properties of them. After, Lewandowska defined generalized $%
2$-normed spaces and derived properties of these spaces in \cite%
{Lewandowska1}. Now, let us give the definition of 2-normed space.

\begin{definition}
Let $X$ be a linear space over $F$, where $F$ is the real or complex numbers
field, $\dim X>1$, and let 
\begin{equation*}
\left\Vert .,.\right\Vert :X^{2}\rightarrow \mathbb{R}^{+}\cup \left\{
0\right\}
\end{equation*}%
be a non-negative real-valued function on $X\times X$ with the following
properties:
\end{definition}

\textit{N1)}\textbf{\ }$\left\Vert x,y\right\Vert =0$ \textit{if and only if}
$x$ \textit{and} $y$ \textit{are linearly dependent vectors},

\textit{N2)}\textbf{\ }$\left\Vert x,y\right\Vert =\left\Vert y,x\right\Vert 
$ \textit{for all} $x,y\in X$,

\textit{N3)}\textbf{\ }$\left\Vert \alpha x,y\right\Vert =\left\vert \alpha
\right\vert \left\Vert x,y\right\Vert $ \textit{for all} $\alpha \in K$ 
\textit{and all }$x,y\in X$,

\textit{N4)}\textbf{\ }$\left\Vert x+y,z\right\Vert \leq \left\Vert
x,z\right\Vert +\left\Vert y,z\right\Vert $ \textit{for all} $x,y,z\in X$.

\textit{Then,} $\left\Vert .,.\right\Vert $ \textit{is called a 2-norm on} $%
X $ \textit{and }$\left( X,\left\Vert .,.\right\Vert \right) $ \textit{is
called a linear 2-normed space}.

Every 2-normed space is a locally convex topological linear space. In fact,
for a fixed $b\in X$. For all $x\in X$, 
\begin{equation*}
p_{b}\left( x\right) =\left\Vert x,b\right\Vert
\end{equation*}

which is a seminorm and the family of $P$, that is 
\begin{equation*}
P=\left\{ p_{b}:b\in X\right\}
\end{equation*}
generates a locally convex topology on $X.$ This space will be denoted by $%
\left( X,p_{b}\right) $. In each 2-normed space $\left( X,\left\Vert
.,.\right\Vert \right) $. For all $x,y\in X$ and for every real $\alpha $,
we have non-negative norm, 
\begin{equation*}
\left\Vert x,y\right\Vert \geq 0\text{ and }\left\Vert x,y+\alpha
x\right\Vert =\left\Vert x,y\right\Vert \text{.}
\end{equation*}
Also, if $x$, $y$ and $z$ are linearly dependent, this occurs for $\dim X=2$%
. Then, 
\begin{equation*}
\left\Vert x,y+z\right\Vert =\left\Vert x,y\right\Vert +\left\Vert
x,z\right\Vert or\left\Vert x,y-z\right\Vert =\left\Vert x,y\right\Vert
+\left\Vert x,z\right\Vert \text{.}
\end{equation*}

\begin{example}
(\cite{White})Let $P_{n}$ denotes the set of real polynomials of degree less
than or equal to $n$, on the interval $\left[ 0,1\right] $. By considering
usual addition and scalar multiplication, $P_{n}$ is a linear vector space
over the reals. Let $\left\{ x_{1},x_{2},\cdots ,x_{2n}\right\} $ be
distinct fixed points in $\left[ 0,1\right] $ and define the 2-norm on $%
P_{n} $ as%
\begin{equation*}
\left\Vert f,g\right\Vert =\dsum\limits_{k=1}^{2n}\left\vert f\left(
x_{k}\right) g^{\prime }\left( x_{k}\right) -f^{\prime }\left( x_{k}\right)
g\left( x_{k}\right) \right\vert \text{.}
\end{equation*}%
Then, $\left( P_{n},\left\Vert .,.\right\Vert \right) $ is a 2-normed space.
\end{example}

Let $\left( X,\left\Vert .,.\right\Vert \right) $ be a 2-normed space. Under
this assumption, we can give the following defitinions:

\begin{definition}
(\cite{White})A sequence $\left\{ x_{n}\right\} _{n\geq 1}$ in a linear
2-normed space $X$ is called Cauchy sequence if there exist independent
elements $y,z\in X$ such that 
\begin{equation*}
\lim_{n,m\rightarrow \infty }\left\Vert x_{n}-x_{m},y\right\Vert =0\text{
and }\lim_{n,m\rightarrow \infty }\left\Vert x_{n}-x_{m},z\right\Vert =0%
\text{.}
\end{equation*}
\end{definition}

\begin{definition}
(\cite{White})A sequence $\left\{ x_{n}\right\} _{n\geq 1}$ in a linear
2-normed space $X$ is called convergent if there exists an element $x\in X$
such that 
\begin{equation*}
\lim_{n\rightarrow \infty }\left\Vert x_{n}-x,z\right\Vert =0
\end{equation*}%
for all $z\in X$.
\end{definition}

\begin{proposition}
(\cite{Cobzas})Let $\left( X,\left\Vert .,.\right\Vert \right) $ be 2-normed
space and $W$ be a subspace of $X$, $b\in X$ and $L\left\{ b\right\} $ be
the subspace of $X$ generated by $b$. If $x_{0}\in X$ is such that%
\begin{equation*}
\delta =\underset{w\in W}{\inf }\left\{ \left\Vert x_{0}-w,b\right\Vert
\right\} >0\text{.}
\end{equation*}%
Then, there exists a bounded bilinear functional as follows 
\begin{equation*}
f:X\times L\left\{ b\right\} \rightarrow K
\end{equation*}
such that 
\begin{equation*}
F|_{w\times L\left\{ b\right\} }=0\text{, }F\left( x_{0},b\right) =1\text{
and }\left\Vert F\right\Vert =\frac{1}{\delta }\text{.}
\end{equation*}
\end{proposition}

\begin{definition}
A 2-normed space $\left( X,\left\Vert .,.\right\Vert \right) $\ is which
every Cauchy sequence $\left( x_{n}\right) $ converges to some $x\in X$ then 
$X$ is said to be complete with respect to the 2-norm.
\end{definition}

\begin{definition}
A complete 2-normed space $\left( X,\left\Vert .,.\right\Vert \right) $\ is
called a 2-Banach space.
\end{definition}

The examples 1 and 2 are 2-Banach spaces while the example 3 does not (For
details, see \cite{White}).

\begin{lemma}
(\cite{White}) $\left( i\right) $ Every 2-normed space of dimension 2 is a
2-Banach space, when the underlying field is complete.
\end{lemma}

$\left( ii\right) $\textit{\ If }$\left\{ x_{n}\right\} $\textit{\ is a
sequence in 2-normed space }$\left( X,\left\Vert .,.\right\Vert \right) $%
\textit{\ and if}%
\begin{equation*}
\lim_{n\rightarrow \infty }\left\Vert x_{n}-x,y\right\Vert =0\text{\textit{.}%
}
\end{equation*}

then, we have%
\begin{equation*}
\lim_{n\rightarrow \infty }\left\Vert x_{n},y\right\Vert =\left\Vert
x,y\right\Vert \text{.}
\end{equation*}

\section{\textbf{Fundamental Results}}

In this section, let us also consider a definition, and however, we give
Lemma and Proposition for the best simultaneous approximation in linear $2$%
-normed spaces.

\begin{definition}
Let $\left( X,\left\Vert .,.\right\Vert \right) $ be a linear $2$-normed
space and $W$ be any bounded subset of $X$. An element $g^{\ast }\in G$ is
said to be a best approximation to the set $W$, if%
\begin{equation*}
\underset{f\in W}{\sup }\left\Vert f-g^{\ast },b\right\Vert =\underset{g\in G%
}{\inf }\left\{ \underset{f\in W}{\sup }\left\Vert f-g,b\right\Vert \right\}
\end{equation*}%
\newline
where $b\in X\backslash L\left\{ f,g^{\ast }\right\} $ is the subspace of $X$
generated by $f$ and $g^{\ast }$.
\end{definition}

\begin{lemma}
Let $\left( X,\left\Vert .,.\right\Vert \right) $ be a linear $2$-normed
space, $G\subset X$ and $W$ be bounded subset of $X$. Then, 
\begin{equation*}
\Phi \left( g,b\right) =\underset{f\in W}{\sup }\left\{ \left\Vert
f_{1}-g,b\right\Vert ,\left\Vert f_{2}-g,b\right\Vert \right\}
\end{equation*}%
is a continuous functional on $X$, where $b\in X\backslash L\left\{
f,g^{\ast }\right\} $.
\end{lemma}

\begin{proof}
Since the norms $\left\Vert f_{1}-g,b\right\Vert ,\left\Vert
f_{2}-g,b\right\Vert $are continuous functionals of $g$\ on $X$, $\phi
\left( g,b\right) $ is clearly a continuous functional. To show this, for
any $f_{1},f_{2}\in W$\ and $g,g^{^{\prime }}\in X$\ , we have 
\begin{eqnarray*}
&&\left\{ \left\Vert f_{1}-g,b\right\Vert ,\left\Vert f_{2}-g,b\right\Vert
\right\} \\
&\leq &\left\{ \left\Vert f_{1}-g^{^{\prime }},b\right\Vert +\left\Vert
g-g^{^{\prime }},b\right\Vert ,\left\Vert f_{2}-g^{^{\prime }},b\right\Vert
,\left\Vert g-g^{^{\prime }},b\right\Vert \right\} \text{.}
\end{eqnarray*}

Then

\begin{eqnarray*}
&&\underset{f_{1},f_{2}\in w}{\sup }\left\{ \left\Vert f_{1}-g,b\right\Vert
,\left\Vert f_{2}-g,b\right\Vert \right\} \\
&\leq &\underset{f_{1},f_{2}\in w}{\sup }\left\{ \left\Vert
f_{1}-g^{^{\prime }},b\right\Vert +\left\Vert g-g^{^{\prime }},b\right\Vert
,\left\Vert f_{2}-g^{^{\prime }},b\right\Vert ,\left\Vert g-g^{^{\prime
}},b\right\Vert \right\} \text{.}
\end{eqnarray*}

Now, if

\begin{equation*}
\left\Vert g-g^{^{\prime }},b\right\Vert <\frac{\varepsilon }{2},\text{ then 
}\phi \left( g,b\right) \leq \phi \left( g^{^{\prime }},b\right)
+\varepsilon \text{.}
\end{equation*}

By interchanging $g$\ and $g^{%
{\acute{}}%
}$, proof of Theorem will be completed.
\end{proof}

\begin{lemma}
Let $\left( X,\left\Vert .,.\right\Vert \right) $ be a linear $2$-normed
space, $G\subset X$ and $W$ be bounded subset of $X$. Then there exists a
best simultaneous approximation $g^{\ast }\in G$ to any given compact subset 
$W\subset X$.
\end{lemma}

\begin{proof}
By using the proof of Elumalai and his coworkers in same manner, we can make
the proof, using the definition of the continuous functional 
\begin{equation*}
\phi \left( g,b\right) =\underset{f\in W}{\sup }\left\{ \left\Vert
f_{1}-g,b\right\Vert ,\left\Vert f_{2}-g,b\right\Vert \right\} \text{.}
\end{equation*}
\end{proof}

\begin{lemma}
Let $\left( X,\left\Vert .,.\right\Vert \right) $ be a linear $2$-normed
space, $G\subset X$ and $W$ be bounded subset of $X$. If $g_{1},g_{2}\in G$\
are best simultaneous approximations to $W$ by elements of $G$. Then $%
g=\lambda _{1}g_{1}+\lambda _{2}g_{2}$ is also a best simultaneous
approximation to $f_{1}$ and $f_{2}$, where $0\leq \lambda \leq 1$ and $%
\lambda _{1}+\lambda _{2}=1$.
\end{lemma}

\begin{proof}
By using expression of $\underset{f\in W}{\sup }\left\{ \left\Vert f_{1}-%
\overset{\_}{g},b\right\Vert ,\left\Vert f_{2}-\overset{\_}{g},b\right\Vert
\right\} $, we discover the followings

\begin{eqnarray*}
&=&\underset{f\in W}{\sup }\left\{ \left\Vert f_{1}-\lambda
_{1}g_{1}-\lambda _{2}g_{2},b\right\Vert ,\left\Vert f_{2}-\lambda
_{1}g_{1}-\lambda _{2}g_{2},b\right\Vert \right\} \\
&=&\underset{f\in W}{\sup }\left\{ \left\Vert \lambda \left(
f_{1}-g_{1}\right) +\left( 1-\lambda \right) \left( f_{1}-g_{2}\right)
,b\right\Vert ,\left\Vert \lambda \left( f_{2}-g_{1}\right) +\left(
1-\lambda \right) \left( f_{2}-g_{2}\right) ,b\right\Vert \right\} \text{.}
\end{eqnarray*}

From last equality, we easily derive as%
\begin{equation*}
\leq \left[ 
\begin{array}{c}
\underset{f\in W}{\sup }\left\{ \left\Vert \lambda \left( f_{1}-g_{1}\right)
+\left( 1-\lambda \right) \left( f_{1}-g_{2}\right) ,b\right\Vert
,\left\Vert \lambda \left( f_{2}-g_{1}\right) +\left( 1-\lambda \right)
\left( f_{2}-g_{2}\right) ,b\right\Vert \right\} \\ 
+\underset{f\in W}{\sup }\left\{ \left\Vert \lambda \left(
f_{1}-g_{1}\right) +\left( 1-\lambda \right) \left( f_{1}-g_{2}\right)
,b\right\Vert ,\left\Vert \lambda \left( f_{2}-g_{1}\right) +\left(
1-\lambda \right) \left( f_{2}-g_{2}\right) ,b\right\Vert \right\}%
\end{array}%
\right]
\end{equation*}

By using definition of 2-norm and definition 10, we deduce as follows%
\begin{equation*}
\underset{g\in G}{\inf }\left\{ \underset{f\in W}{\sup }\left\Vert
f_{1}-g,b\right\Vert ,\left\Vert f_{2}-g,b\right\Vert \right\} \text{.}
\end{equation*}

Subsequently, we complete the proof of Lemma.
\end{proof}

\begin{proposition}
Let $\left( X,\left\Vert .,.\right\Vert \right) $ be a linear $2$-normed
space, $G$ is a non-empty strictly convex subset of $X$ and $Y$ be a compact
subset of $X$. Then there is only one $y_{0}\in Y$ such that 
\begin{equation*}
\left\Vert x_{0}-y_{0},z\right\Vert =\underset{y\in Y}{\inf }\left\{
\left\Vert x_{0}-y,z\right\Vert \right\}
\end{equation*}%
for $x_{0}\in X\backslash Y$ and for every $z\in X\backslash L\left\{ x\in
G\right\} $ and $y_{0}\in Y$.
\end{proposition}

\begin{proof}
If $x_{0}\in Y$. Then, we have $\left\Vert x_{0}-y_{0},z\right\Vert =0$.
Hence, assume that 
\begin{equation*}
x_{0}\in Y\ \text{or }x\in X\backslash Y\text{.}
\end{equation*}

If we say

\begin{equation*}
d_{0}=\underset{y\in Y}{\inf }\left\{ \left\Vert x_{0}-y,z\right\Vert
\right\}
\end{equation*}

and 
\begin{equation*}
d_{0}=\underset{y\in Y}{\inf }\left\{ \left\Vert x_{0}-y,y^{^{\prime
}}\right\Vert \right\} \text{.}
\end{equation*}

Then, there are linearly independent elements $y^{^{\prime }}$\ and $z$\ in $%
X$. So, there is a Cauchy sequence $\left\{ y_{n}\right\} $\ such that

\begin{equation*}
\underset{n\rightarrow \infty }{\lim }\left\Vert x_{0}-y_{n},z\right\Vert
=d_{0},\text{ \ \ }\underset{m\rightarrow \infty }{\lim }\left\Vert
x_{0}-y_{m},z\right\Vert =d_{0}
\end{equation*}

and 
\begin{equation*}
\underset{n\rightarrow \infty }{\lim }\left\Vert x_{0}-y_{n},y^{^{\prime
}}\right\Vert =d_{0},\text{ \ \ }\underset{m\rightarrow \infty }{\lim }%
\left\Vert x_{0}-y_{m},y^{^{\prime }}\right\Vert =d_{0}\text{.}
\end{equation*}

Thus, we procure the following 
\begin{equation*}
\left\Vert x_{0}-y_{0},y\right\Vert =d_{0}\text{ \ and \ }\left\Vert
x_{0}-y_{0},z\right\Vert =d_{0}\text{.}
\end{equation*}

By using the following inequalities

\begin{equation*}
d_{0}\leq \left\Vert x_{0}-y_{0},z\right\Vert \leq \left\Vert
x_{0}-y_{n},z\right\Vert +\left\Vert y_{n}-y_{0},z\right\Vert
\end{equation*}

and%
\begin{equation*}
d_{0}\leq \left\Vert x_{0}-y_{0},y\right\Vert \leq \left\Vert
x_{0}-y_{n},y\right\Vert +\left\Vert y_{n}-y_{0},y\right\Vert \text{.}
\end{equation*}

From this, we see that

\begin{eqnarray*}
\left\Vert y_{0}-y_{0}^{^{\prime }},z\right\Vert ^{2} &=&\left\Vert \left(
y_{0}-x_{0}\right) +\left( x_{0}-y_{0}^{^{\prime }}\right) ,z\right\Vert ^{2}
\\
&=&\left\Vert \left( y_{0}-x_{0}\right) +\left( x_{0}-y_{0}^{^{\prime
}}\right) ,z\right\Vert ^{2}+\left\Vert \left( y_{0}-x_{0}\right) +\left(
x_{0}-y_{0}^{^{\prime }}\right) ,z\right\Vert ^{2} \\
&&-\left\Vert y_{0}+y_{0}^{^{\prime }}-2x_{0},z\right\Vert ^{2} \\
&\leq &2\left( \left\Vert x_{0}-y_{0},z\right\Vert ^{2}+\left\Vert
y_{0}^{^{\prime }}-x_{0},z\right\Vert ^{2}\right) -4\left\Vert \frac{%
y_{0}+y_{0}^{^{\prime }}}{2}-x_{0},z\right\Vert ^{2} \\
&\leq &2\left( 2d_{0}^{2}\right) -4d_{0}^{2}=0\text{.}
\end{eqnarray*}

We find $y_{0}=y_{0}^{^{\prime }}$. In similar way, we again obtain $%
y_{0}=y_{0}^{^{\prime }}$\ with respect to $y$

\begin{eqnarray*}
\left\Vert y_{0}-y_{0}^{^{\prime }},y\right\Vert ^{2} &=&\left\Vert \left(
y_{0}-x_{0}\right) +\left( x_{0}-y_{0}^{^{\prime }}\right) ,y\right\Vert ^{2}
\\
&=&\left\Vert \left( y_{0}-x_{0}\right) +\left( x_{0}-y_{0}^{^{\prime
}}\right) ,y\right\Vert ^{2}+\left\Vert \left( y_{0}-x_{0}\right) +\left(
x_{0}-y_{0}^{^{\prime }}\right) ,y\right\Vert ^{2} \\
&&-\left\Vert y_{0}+y_{0}^{^{\prime }}-2x_{0},y\right\Vert ^{2} \\
&\leq &2\left( \left\Vert x_{0}-y_{0},y\right\Vert ^{2}+\left\Vert
y_{0}^{^{\prime }}-x_{0},y\right\Vert ^{2}\right) -4\left\Vert \frac{%
y_{0}+y_{0}^{^{\prime }}}{2}-x_{0},y\right\Vert ^{2} \\
&\leq &2\left( 2d_{0}^{2}\right) -4d_{0}^{2}=0
\end{eqnarray*}

Then, we complete the proof of theorem.
\end{proof}

Let $X$\ be a linear 2-normed space and $W_{1}$\ and $W_{2}$\ are linear
subspaces in $X$, and $f$\ be a 2-functional with domain $W_{1}\times W_{2}$%
. If $\left\Vert .,.\right\Vert $\ denotes 2-norm, then the problem is to
find an element $g^{\ast }\in G$, if it exists for which

\begin{equation*}
\underset{f_{1},f_{2}\in W}{\sup }\left\{ \left\Vert f_{1}-g^{\ast
},b\right\Vert ,\left\Vert f_{2}-g^{\ast },b\right\Vert \right\} =\underset{%
g\in G}{\inf }\left\{ \underset{f_{1},f_{2}\in W}{\sup }\left( \left\Vert
f_{1}-g^{\ast },b\right\Vert ,\left\Vert f_{2}-g^{\ast },b\right\Vert
\right) \right\} \text{.}
\end{equation*}

Thus, we reach the following proposition which is interesting and worthwhile
for studying in linear $2$-normed spaces.

\begin{proposition}
Let $\left( X,\left\Vert .,.\right\Vert \right) $\ be a linear 2-normed
space over $R$\ and $G$\ be a linear subspace of $X$. Let $f_{1},f_{2}\in
X\backslash G$ such that $f_{1},f_{2}$ and $b\in X$ are linearly
independent. Then there exists a best simultaneous approximation by elements
of $G$ to $f_{1},f_{2}\in W$ such that%
\begin{equation*}
\underset{g\in G}{\inf }\left\{ \underset{f_{1},f_{2}\in W}{\sup }\left(
\left\Vert f_{1}-g,b\right\Vert ,\left\Vert f_{2}-g,b\right\Vert \right)
\right\} =\underset{f_{1},f_{2}\in W}{\sup }\left\{ \left\Vert f_{1}-g^{\ast
},b\right\Vert ,\left\Vert f_{2}-g^{\ast },b\right\Vert \right\}
\end{equation*}%
where $W=\left\{ f_{1},f_{2}\right\} $.
\end{proposition}

\begin{acknowledgement}
Author would like to thank to Serkan Araci for his help in this paper.
\end{acknowledgement}

\end{document}